\documentclass[12pt,amssymb]{article}
\usepackage[russian]{babel}
\usepackage{amsmath}
\usepackage{amssymb}
\usepackage{amsfonts}
\usepackage{amsthm}
\usepackage[T2A]{fontenc}
\begin{document}

\newtheorem{thm}{Теорема}
\newtheorem{lm}{Утверждение}
\newtheorem*{zam}{Замечание}
\newtheorem*{sss}{Следствие}
\theoremstyle{definition}
\newtheorem*{zakzam}{Заключительные замечания}
\theoremstyle{definition}
\newtheorem{df}{Определение}
\newtheorem*{blagodar}{Благодарности}

\author{\textbf{Н.~Л.~Поляков}}
\title{\textbf{Об алгоритмической разрешимости проблемы распознавания \\бесквадратности данного
слова относительно системы из двух \\определяющих соотношений}}
\date{}
\maketitle

\begin{abstract}Проблема распознавания бесквадратности данного слова относительно произвольной системы из двух определяющих соотношений
алгоритмически разрешима.
\end{abstract}

\par Пусть $A$ \--- непустое рекурсивное множество, называемое \emph{алфавитом}, элементы
 которого называются \emph{буквы}. \emph{Словом} в алфавите $A$
 называется конечная последовательность \mbox{$x_{1}x_{2}\ldots x_{n}$}, \mbox{$0 \leq n < \omega$}, элементов
 из множества $A$. Натуральное число $n$ называется длиной слова \mbox{$x=x_{1}x_{2}\ldots x_{n}$} и обозначается символом \mbox{$|x|$}.
 Последовательность длины $0$ является пустым множеством, поэтому пустое слово будет обозначаться символом
 $\emptyset$. Множество всех слов в алфавите $A$ обозначается символом $A^\ast$. На этом множестве определена бинарная операция $"$приписывания$"$, называемая
 иногда \emph{конкатенация}, ставящая в соответствие паре слов
  \mbox{$(\, x_{1}x_{2} \ldots x_{n}, \; y_{1}y_{2} \ldots y_{m} \,)$} слово \mbox{$x_{1}x_{2} \ldots x_{n}y_{1}y_{2} \ldots
 y_{m}$} и наделяющая множество $A^\ast$ структурой полугруппы с единицей, которая называется свободным моноидом на множестве
 $A$. Мы будем обозначать ее тем же символом $A^\ast$.
 \par Слово $w$ называется \emph{бесквадратным}, если $$w = ussv
 \Rightarrow s = \emptyset$$ для всех слов $u, s, v$ из $A^\ast$. Множество всех бесквадратных слов в алфавите A мы будем обозначать $SF(A)$.
 Согласно
 известному результату А. Туэ (\,см.~\cite{lol}\,), полученном в $1906$ году, если алфавит $A$ содержит по
 крайней мере три буквы, то множество $SF(A)$ бесконечно (для одно-
 и двухбуквенных алфавитов, как легко проверить, верно обратное). Этот результат А. Туэ неоднократно
 переоткрывался и передоказывался другими авторами (\,см., напр.~\cite{archon}\,). Обобщения этого результата и
 решение родственных задач можно найти в ~\cite{Brink}, ~\cite{Karpi}, ~\cite{Karpi2},~\cite{Karpi4} и др.
 В ряде работ исследовано понятие
 \emph{бесквадратности относительно системы определяющих
 соотношений}.\par
 Системой определяющих соотношений называется произвольное бинарное
 иррефлексивное отношение $\pi$ на множестве $A^\ast$.
 Системы определяющих соотношений часто записывают в виде множества
 "равенств" \; \mbox{$u = v$}, однако в данной работе мы будем использовать
 обозначение \mbox{$u = v$}, где \mbox{$u,v \in A^{\ast} $} только для обозначения равенства в свободном моноиде,
 т.е. графического совпадения слов.\par
Для данной системы определяющих соотношений $\pi$ на множестве
$A^\ast$ следующим образом определяется бинарное отношение
$\stackrel{\pi}{\leftrightarrow}$ \emph{непосредственной
выводимости}: \mbox{$x\stackrel{\pi}{\leftrightarrow}y$} тогда и
только тогда, когда \mbox{$x = rus \; \& \; y = rvs \; \& \;
(\,(\,u, v\,)\in\pi \: \vee \: (\,v, u\,)\in\pi \,)$} для некоторых
слов $r, s, u, v$ из $A^\ast$. Рефлексивное и транзитивное замыкание
отношения $\stackrel{\pi}{\leftrightarrow}$ называется равенством
относительно системы определяющих соотношений $\pi$ и обозначается
символом $\stackrel{\pi}{=}$. Это отношение является конгруэнцией
моноида $A^\ast$. Класс эквивалентности слова \mbox{$w\in A^\ast$}
обозначается символом $[w]_\pi$. Пара $(A, \pi)$ иногда называется
копредставлением любого моноида, изоморфного фактор-моноиду моноида
$A^\ast$ по конгруэнции $\stackrel{\pi}{=}$.
\par Слово \mbox{$w\in A^\ast$} называется \emph{бесквадратным относительно
системы соотношений} $\pi$, если \mbox{$[w]_\pi\subseteq SF(A)$}.
Множество всех таких слов будем обозначать \mbox{$SF(A,\pi)$}. Если
множество $SF(A,\pi)$ рекурсивно, то говорят, что \emph{разрешима
проблема распознавания бесквадратности данного слова относительно
системы определяющих соотношений} $\pi$. В общем случае эта проблема
неразрешима (см.~\cite{Karpi}). Однако, имеет место следующая

\begin{thm}Пусть даны алфавит $A$ и  система определяющих соотношений
\mbox{$\pi=\{\,(\, u_i, u_j \,)\mid 1 \leq i < j \leq n \, \}$}, где
\mbox{$n < \omega$}, \mbox{$u_i \in A^\ast$} и
\mbox{$u_i=u_j\Rightarrow i=j$} для всех $i$ и $j$, \mbox{$1\leq
i\leq n,\; 1\leq j\leq n$}. Тогда множество $SF(A,\pi)$ рекурсивно.
\end{thm}

В частности, разрешима проблема распознавания бесквадратности
данного слова относительно одноэлементной системы определяющих
соотношений $\{\,(\, u,v \, )\, \}$, $u\neq v$.
\par Эти результаты
доказаны  A. Карпи, A. де Лука в~\cite{Karpi}. Ниже будет доказано,
что относительно двухэлементной системы определяющих соотношений эта
проблема также разрешима.
\par Последние результаты представляют
интерес, в частности, в связи с открытостью вопроса о
\emph{разрешимости проблемы равенства относительно произвольной
одно- и двухэлементной системы определяющих соотношений}, т.е. о
рекурсивности отношений \mbox{$x\stackrel{\pi}{=}y$} и
\mbox{$x\stackrel{\pi}{=}w$} для данного слова $w$, где $\pi$ имеет
вид $\{\,(u,v)\, \}$ или $\{\,(u, v), (r, s)\, \}$ (\,пример системы
из \emph{трех} определяющих соотношений с неразрешимой проблемой
равенства привел Матиясевич~\cite{mati}; см. также обзор
в~\cite{lol}\,). \par На самом деле в работе ~\cite{Karpi} доказано
более сильное, чем теорема $1$ утверждение. А именно, доказано, что
для любой системы определяющих соотношений $\pi$ вида \mbox{$\{\,(\,
u_i, u_j \,)\mid 1 \leq i < j \leq n \, \}$}, \mbox{$n<\omega$},
выполнено следующее условие : $$ \forall w \in A^\ast \; w \in
SF(A,\pi) \Rightarrow \|[w]_\pi\| < \omega.\eqno (\ast)
$$
Из этого утверждения, конечно, следует теорема $1$. Действительно,
для любого рекурсивного множества \mbox{$P(A) \subseteq A^\ast$} его
$"$релятивизация$"$ относительно системы соотношений $\pi$, т.е.
множество \mbox{$P(A, \pi) = \{\,u \in A^\ast \mid [u]_\pi \subseteq
P(A) \, \}$} также рекурсивно, если каждый его элемент $w$ имеет
конечный класс эквивалентности относительно системы соотношений
$\pi$. Для доказательства этого факта можно рассмотреть для любого
слова \mbox{$w \in A^\ast$} последовательность конечных множеств
слов \mbox{$U_1, U_2,\ldots U_i,\ldots$} , где \mbox{$U_1 = \{\, w
\, \}$}, а \mbox{$U_{i+1} = U_i \cup \{ \, u \in A^\ast \mid \exists
v\in U_i\; v\stackrel{\pi}{\leftrightarrow}u\, \}$}. Алгоритм,
осуществляющий последовательную проверку принадлежности множеству
$P(A)$ каждого элемента очередного множества $U_i$ до тех пор, пока
не будет получен отрицательный ответ или $U_i$ не совпадет с
$U_{i-1}$, разрешает вопрос о принадлежности слова $w$ множеству
$P(A, \pi)$ поскольку конечность множества
\mbox{$[w]_{\pi}=\bigcup_{i<\omega}U_{i}$} гарантирует, что одно из
этих событий произойдет на некотором конечном шаге.
\begin{thm} Пусть даны алфавит $A$ и  система определяющих соотношений \mbox{$\pi = \{\, (\, u_i, u_j \, )\mid 1\leq i < j \leq n\,\}\cup
\{\, (\, v_k, v_l \,) \mid 1\leq k < l\leq m \,\}$}, где \mbox{$n,m
< \omega$}, \mbox{$u_i,v_k \in A^\ast$} и \mbox{$u_i = u_j
\Rightarrow i = j\,\&\, v_k = v_l \Rightarrow k = l$} для всех
номеров \mbox{$i,j,k,l$}, \mbox{$1\leq i,j \leq n$}, \mbox{$1\leq
k,l \leq n$}. Тогда выполнено условие $(\ast)$.
\end{thm}
\begin{proof}[Доказательство]
Для каждой системы определяющих соотношений $\theta$ в алфавите $A$
множество всех слов, входящих в определяющие соотношения, называется
множеством определяющих слов и обозначается символом $D_\theta$:
\mbox{$x \in D_\theta$} тогда и только тогда, когда \mbox{$\exists y
\in A^\ast \; ((x, y) \in \theta \vee (y, x) \in \theta)$}.
Обозначим через $\sigma$ систему соотношений \mbox{$\{\, ( \, u_i,
u_j \,) \mid 1\leq i < j \leq n \,\}$}, а через $\rho$ систему
соотношений \mbox{$\{\, (\, v_k, v_l \,) \mid 1\leq k < l\leq m
\,\}$}. Обозначим символом $\overline{D}_\sigma$ замыкание множества
$D_\sigma$ по отношению $\stackrel{\rho}{=}$, а символом
$\overline{D}_\rho$ замыкание множества $D_\rho$ по отношению
$\stackrel{\sigma}{=}$. Систему определяющих соотношений
\mbox{$(\,\overline{D}_\sigma \times \overline{D}_\sigma \, \cup \,
\overline{D}_\rho \times \overline{D}_\rho\,) \, \setminus \,
\{\,(\, u, u\, )\mid u \in A^{\ast} \, \}$} обозначим $\tau$.
Очевидно, отношения $\stackrel{\pi}{=}$ и $\stackrel{\tau}{=}$ на
множестве $A^\ast$ совпадают. Для сокращения записи вместо выражения
\mbox{$(\, x, y \,) \in \overline{D}_\sigma \times
\overline{D}_\sigma \, \cup \, \overline{D}_\rho \times
\overline{D}_\rho$} будем использовать знакосочетание \mbox{$x \sim
y$}.
\par Следующее утверждение легко следуют из
выполненности условия $(\ast)$ для систем определяющих соотношений
вида \mbox{$ \{\, (\, u_i, u_j \, )\mid 1\leq i < j \leq n\}$}.
\begin{lm} Условие $(\ast)$ выполнено, если имеет место
один из следующих фактов:\\
1. Хотя бы одно из множеств $\overline{D}_\sigma$,
$\overline{D}_\rho$ бесконечно.\\
2. Хотя бы одно из множеств $\overline{D}_\sigma$ и
$\overline{D}_\rho$ не замкнуто по отношению $\stackrel{\sigma}{=}$
и $\stackrel{\rho}{=}$ соответственно.\\
3. \mbox{$D_\sigma=\emptyset$} или \mbox{$D_\rho=\emptyset$} или
\mbox{$\overline{D}_\sigma\, \cap \, \overline{D}_\rho \neq
\emptyset$}.
\end{lm}
\par Действительно, пусть слово $w$ бесквадратно относительно системы соотношений $\pi$.
\par Допустим, что множество $\overline{D}_\sigma$ бесконечно.
Если класс $[w]_\pi$ не содержит ни одного слова \mbox{$w' \in
A^\ast D_\sigma A^\ast$}, то \mbox{$[w]_\pi = [w]_\rho$}, и
\mbox{$\|[w]_\pi\| < \omega$} ввиду очевидного включения
\mbox{$SF(A, \pi)\subseteq SF(A, \rho)$}. Противоположный случай
невозможен в силу следующей цепочки следствий: \mbox{$\|[w']_\rho\|
= \omega \Rightarrow w' \notin SF(A, \rho) \Rightarrow w \notin
SF(A, \pi)$}. Аналогично рассматривается симметричный случай
бесконечности множества $D_{\rho}$, что доказывает первую часть
утверждения.
\par
Пусть теперь множество $\overline{D}_\sigma$ не замкнуто по
отношению $\stackrel{\sigma}{=}$. Тогда существуют слова $r$ и $s$
из $A^\ast$ и слово $u$ из $D_\sigma$, для которых имеет место
включение \mbox{$rus \in \overline{D}_\sigma$}, причем слово
\mbox{$rs$} непусто. Отсюда \mbox{$u \stackrel{\pi}{=} rus
\stackrel{\pi}{=} rruss \notin SF(A)$} и, следовательно,
\mbox{$A^\ast D_\sigma A^\ast \cap SF(A, \pi) = \emptyset$}. Поэтому
\mbox{$[w]_\pi = [w]_\rho$}. Симметричный случай рассматривается
аналогично, что доказывает вторую часть утверждения.
\par Третью часть утверждения достаточно
доказать в предположении конечности множеств $\overline{D}_\sigma$ и
$\overline{D}_\rho$. В этом случае она сразу следует из того, что
равенство относительно системы соотношений $\pi$ совпадает с
равенством относительно системы соотношений \mbox{$\{\,(\, w_i, w_j
\,)\mid 1 \leq i < j \leq k \, \}$}, где множество \mbox{$\{\,w_{i}
\mid 1\leq i \leq k \, \}$} есть \mbox{$\overline{D}_\sigma \cup
\overline{D}_\rho$}. \qed
\par
В дальнейшем будем считать, что множества $\overline{D}_\sigma$ и
$\overline{D}_\rho$ конечны, непусты, замкнуты по отношениям
соответственно $\stackrel{\sigma}{=}$ и $\stackrel{\rho}{=}$, и
имеют пустое пересечение. В частности, отсюда следует, что множество
$D_{\pi}$ не содержит пустого слова (в противном случае каждое из
множеств $\overline{D}_\sigma$ и $\overline{D}_\rho$ бесконечно), а
множество \mbox{$D_{\tau}$} замкнуто по отношению
$\stackrel{\pi}{=}$.
\begin{df}
Последовательность непустых слов \mbox{$x_{1}, x_{2}, \ldots ,
x_{n}$}, \mbox{$1 \leq n < \omega$}, назовем линейным разложением
слова \mbox{$x = x_{1}x_{2} \ldots x_{n}$}, если существуют такие
слова $ p_i, \, q_i, \, u_i, \, v_i$, \mbox{$1 \leq i \leq n$}, что
выполнены следующие условия:\begin{itemize}
\item[$(1)$] $p_{i}x_{i}q_{i} \sim q_{i-1}u_{i}$ для всех номеров $i$, $1 < i \leq n$,
\item[$(2)$] $p_{i}x_{i}q_{i} \sim v_{i}p_{i+1}$ для всех номеров $i$, $1 \leq i < n$,
\item[$(3)$] $q_i = \emptyset \vee p_{i+1} = \emptyset$ для всех номеров $i$, $1 \leq i < n$,
\item[$(4)$] $p_{1} = q_{n} = \emptyset$,
\item[$(5)$] $n = 1 \Rightarrow x_{1} \in D_{\tau}$.
\end{itemize}
\end{df}
Для любого номера $i$, \mbox{$1 \leq i \leq n$}, слова $p_i$ и $q_i$
будем называть соответственно $i$-ым левыми и $i$-ым правым
дополнительным членом линейного разложения \mbox{$x_{1}, x_{2},
\ldots, x_{n}$} слова $x$. Слова, имеющие линейное разложение с
числом членов не более $n$, назовем линейно разложимыми порядка $n$.
Множество всех линейно разложимых слов порядка $n$ будем обозначать
символом $Lin(n)$. Положим \mbox{$Lin = \bigcup\limits_{\, n <
\omega \, } Lin(n)$}. Элементы множества $Lin$ будем называть
линейно разложимыми словами.
\par
\begin{lm}Пусть последовательность \mbox{$x_{1}, x_{2},
\ldots , x_{n}$}, \mbox{$1 \leq n < \omega$}, есть линейное
разложение слова $x$, и слова $p_i$ и $q_i$, \mbox{$1 \leq i \leq n
$}, суть $i$-ые левые и $i$-ые правые дополнительные члены данного
линейного разложения. Тогда: \begin{itemize} \item[1.] Для каждого
номера $i$, \mbox{$1 \leq i \leq n$}, существуют такие слова $f$ и
$g$, что выполнены равенства
\\ $x_{1}x_{2}\ldots x_{i-1} \stackrel{\pi}{=}
fp_{i}$ и $x_{i+1}x_{i+2}\ldots x_{n} \stackrel{\pi}{=} q_{i}g$.
\item[2.] Для каждого номера $i$, \mbox{$1 \leq i \leq n$},
последовательность слов \mbox{$x_{1}, x_{2}, \ldots , x_{i}q_{i}$},
есть линейное разложение слова \mbox{$x_{1}
x_{2} \ldots x_{i}q_{i}$},\\
а последовательность слов \mbox{$p_{i}x_{i}, x_{i+1}, \ldots ,
x_{n}$} есть линейное разложение слова \mbox{$p_{i}x_{i}x_{i+1}
\ldots x_{n}$}.
\item[3.] Если для некоторого номера $i$, \mbox{$1 \leq i \leq n$},
выполнено равенство \mbox{$x_{i} \stackrel{\pi}{=} x'$}, то
последовательность \mbox{$x_{1}, x_{2}, \ldots , x_{i-1}, x',
x_{i+1}, \ldots , x_{n}$} есть линейное разложение слова
\mbox{$x_{1}x_{2}\ldots x_{i-1}x'x_{i+1} \ldots x_{n}$}.
\item[4.] Пусть также последовательность слов \mbox{$y_{1}, y_{2},
\ldots , y_{m}$}, \mbox{$1 \leq m < \omega$}, есть линейное
разложение слова $y$, и для некоторых слов $e, x', y'$ выполнено:
\mbox{$x_{n} = x'e \; \& \; y_{1} = ey'$}. Тогда слово \mbox{$z =
x_{1}x_{2} \ldots x_{n-1}x'ey'y_{2} \ldots y_{m-1}y_{m}$}
принадлежит множеству \mbox{$ Lin(\, n+m+sign(\, |x'y'|\,)-1 \, )$},
символ $sign(t)$ обозначает функцию из $\omega$ в \mbox{$\{\, 0, 1
\,\}$}, равную нулю при \mbox{$t = 0$}, и единице иначе.
\par При этом в качестве линейного разложения слова $z$ можно взять последовательность: \\ \mbox{$x_{1}, x_{2},
\ldots , x_{n-1}, x', y_{1}, \ldots , y_{m-1} , y_{m}$},\; если
слово $x'$ непусто, \\ \mbox{$x_{1}, x_{2}, \ldots , x_{n}, y',
y_{2}, \ldots , y_{m-1} , y_{m}$}\; если слово $y'$ непусто, и
\\\mbox{$x_{1}, x_{2}, \ldots , x_{n-1}, e , y_{2},
\ldots , y_{m-1} , y_{m}$}\; если \mbox{$x' = y' = \emptyset$}.
\item[5.] При любом натуральном $n$ множество $Lin(n)$ конечно.
\end{itemize}
\end{lm}
\begin{proof}\emph{1.} Индукцией по $i$. При \mbox{$i=1$} слова \mbox{$x_{1}x_{2}\ldots x_{i-1}
$} и $p_i$ пусты (последнее по определению $1$), поэтому пустое
слово $f$ удовлетворяет условию. Пусть \mbox{$i\geq 2$}, и
утверждение доказано для всех \mbox{$j<i$}. Если слово $q_{i-1}$ не
пусто, то \mbox{$p_i=\emptyset$} по определению $1$, и можно
положить \mbox{$f=x_1x_2\ldots x_{i-1}$}. В противном случае по
предположению индукции для некоторого слова $f'$ имеет место
равенство \mbox{$x_1x_2\ldots x_{i-1}
\stackrel{\pi}{=}f'p_{i-1}x_{i-1}$}, откуда по определению $1$ имеем
\mbox{$x_1x_2\ldots x_{i-1}\stackrel{\pi}{=}f'v_{i-1}p_{i}$} для
некоторого слова $v_{i-1}$, и можно положить \mbox{$f=f'v_{i-1}$}.
Второе равенство доказывается симметрично.
\\ \emph{2.} Тривиальной проверкой: для
всех номеров \mbox{$j\leq i$} $j$-ые левые и правые дополнительные
члены линейного разложения \mbox{$x_{1}, x_{2}, \ldots ,
x_{i}q_{i}$} можно положить равным соответствующим дополнительным
членам линейного разложения \mbox{$x_{1}, x_{2}, \ldots , x_{n}$},
кроме $i$-го правого дополнительного члена, который надо положить
равным пустому слову. Симметрично для последовательности
\mbox{$p_{i}x_{i}, x_{i+1}, \ldots , x_{n}$}.
\\ \emph{3.} Следует из замкнутости множества \mbox{$D_{\tau}$}
по отношению $\stackrel{\pi}{=}$ (левые и правые члены линейного
разложения \ldots{$x_1, x_2, \ldots, x_{i-1}, x',x_{i+1},\ldots,
x_{n}$} можно положить равными соотвествующим левым и правым членам
линейного разложения \mbox{$x_1,x_2,\ldots, x_{n}$}).
\\ \emph{4.} Обозначим для каждого линейного разложения $\xi$ некоторого слова
$x$ символами $L(\xi)$ и $R(\xi)$ последовательности соответственно
левых и правых дополнительных членов. Для доказательства данного
пункта утверждения нужно для каждой последовательности $\xi$ из его
формулировки предъявить последовательности $L(\xi)$ и $R(\xi)$ и
доказать выполненность условий определения $1$. Пусть
$$
L(x_1,x_2,\ldots,x_n)=(p_1,p_2,\ldots, p_n), \,
R(x_1,x_2,\ldots,x_n)=(q_1,q_2,\ldots, q_n),
$$
$$
L(y_1,y_2,\ldots,y_m)=(p'_1,p'_2,\ldots, p'_m), \,
R(y_1,y_2,\ldots,y_m)=(q'_1,q'_2,\ldots, q'_m).
$$
Тогда если слово $x'$ непусто, можно положить
$$
L(x_{1}, x_{2}, \ldots , x_{n-1}, x', y_{1}, \ldots , y_{m-1} ,
y_{m})= (p_1, p_2,\ldots,p_n,p'_1,p'_2,\ldots,p'_m),
$$
$$
R(x_{1}, x_{2}, \ldots , x_{n-1}, x', y_{1}, \ldots , y_{m-1} ,
y_{m})= (q_1, q_2,\ldots,q_{n-1}, e, q'_1,q'_2,\ldots,q'_m).
$$
Поскольку слово $p'_1$ пусто, условие $(3)$ определения $1$
выполнено для номера \mbox{$i=n$}. Кроме того, слова
\mbox{$p_{n}x_{n}$} и \mbox{$y^{}_{1}q'_1$} принадлежат множеству
\mbox{$D_{\tau}$}, откуда следует, что \mbox{$p'_1y^{}_{1}q'_1\sim
ey'q'_1$} и \mbox{$p_{n}x'e\sim p_{n}x'ep'_1$}, что показывает
выполненность условия $(1)$ определения $1$ для номера
\mbox{$i=n+1$} и условия $(2)$ для номера \mbox{$i=n$}. Слово $q_n$
пусто, поэтому \mbox{$p_nx'e=p_nx_nq_n\sim q_{n-1}v_n$} для
некоторого слова $v_n$, что показывает выполненность условия $(1)$
определения $1$ для номера \mbox{$i=n$}. Для всех остальных номеров
выполненность каждого из условий определения $1$ гарантируется
выполненностью соответствующего условия для линейных разложений
\mbox{$x_1,x_2,\ldots,x_n$} и \mbox{$y_1,y_2,\ldots,y_n$}.
\par Если
слово $x''$ непусто, можно положить
$$
L(x_{1}, x_{2}, \ldots , x_{n}, y', y_{2}, \ldots , y_{m-1} ,
y_{m})= (p_1, p_2,\ldots,p_n,e,p'_2,\ldots,p'_m),
$$
$$
R(x_{1}, x_{2}, \ldots , x_{n}, y', y_{2}, \ldots , y_{m-1} ,
y_{m})= (q_1, q_2,\ldots,q_{n}, q'_1,q'_2,\ldots,q'_m).
$$
Доказательство выполненности условий определения $1$ аналогично.
\par Наконец, если \mbox{$x' = y' = \emptyset$} можно положить
$$
L(x_{1}, x_{2}, \ldots , x_{n-1}, e , y_{2}, \ldots , y_{m-1} ,
y_{m})= (p_1, p_2,\ldots,p_n,p'_2,\ldots,p'_m),
$$
$$
R(x_{1}, x_{2}, \ldots , x_{n-1}, e , y_{2}, \ldots , y_{m-1} ,
y_{m})= (q_1, q_2,\ldots,q_{n}, q'_2,\ldots,q'_m).
$$
Поскольку слово $q_n$ пусто, условие $(3)$ определения $1$ выполнено
для номера \mbox{$i=n$}. Слово \mbox{$p'_2y^{}_2q'_2$} принадлежит
множеству \mbox{$D_{\tau}$}, откуда следует, что
\mbox{$p'_2y^{}_2q'_2\sim q^{}_np'_2y^{}_2q'_2$}, что показывает
выполненность условия $(1)$ определения $1$ для номера
\mbox{$i=n+1$}. Слово \mbox{$p_nx_nq_n=p_ne$} принадлежит множеству
\mbox{$D_{\tau}$}, поэтому если слово \mbox{$p'_{2}$} пусто, то
\mbox{$p_neq_n\sim p_neq^{}_np'_{2}$}. Если же \mbox{$p'_{2}\neq
\emptyset$}, то \mbox{$q'_{1}=\emptyset$} и слово \mbox{$e=y_1$}
принадлежит множеству \mbox{$D_{\tau}$}, причем для некоторого слова
$v'_1$ выполнено \mbox{$e\sim v'_1p'_2$}. Тогда \mbox{$p_neq_n\sim
p^{}_{n}v'_1p'_2$} в силу замкнутости множества $\mbox{$D_{\tau}$}$
по отношению $\stackrel{\pi}{=}$. Таким образом в обоих случаях
выполнено условие $(2)$ определения $1$ для номера \mbox{$i=n$}. Для
всех остальных номеров выполненность каждого из условий определения
$1$ гарантируется выполненностью соответствующего условия для
линейных разложений \mbox{$x_1,x_2,\ldots,x_n$} и
\mbox{$y_1,y_2,\ldots,y_n$}.
\\ \emph{4.} Следует из конечности множеств $\overline{D}_\sigma$ и
$\overline{D}_\rho$.
\end{proof}
\begin{lm} Пусть последовательность \mbox{$x_{1}, x_{2},
\ldots , x_{n}$}, $1 \leq n < \omega$, есть линейное разложение
слова \mbox{$x\in SF(A, \pi)$}, а слова $p_i$ и $q_i$, \mbox{$1 \leq
i \leq n $}, суть $i$-ые левые и $i$-ые правые дополнительные
члены данного линейного разложения.\\
Тогда для любого номера $i$, $1\leq i < n$, выполнено \mbox{$\neg
(p_{i}x_{i}q_{i} \sim p_{i+1}x_{i+1}q_{i+1})$}.
\end{lm}
\begin{proof}[Доказательство]
Предположим обратное. Тогда по утверждению $2.1$ для некоторых слов
$f$ и $g$ имеет место равенство \mbox{$x\stackrel{\pi}{=}
fp_{i}x_{i}x_{i+1}q_{i+1}g$}. Если \mbox{$p_{i+1}\neq \emptyset$},
то \mbox{$q_i =\emptyset$}, что влечет:
$$x\stackrel{\pi}{=}
fp_{i+1}\underbrace{x_{i+1}q_{i+1}x_{i+1}q_{i+1}}g\notin SF(A,
\pi).$$ В противном случае $$x\stackrel{\pi}{=}
f\underbrace{p_{i}x_{i}p_{i}x_{i}}q_{i+1}g\notin SF(A, \pi).$$
Противоречие.
\end{proof}
\begin{lm} Пусть последовательность \mbox{$x_{1}, x_{2},
\ldots , x_{n}$}, \mbox{$1 \leq n < \omega$}, есть линейное
разложение слова $x$ и слова $p_i$ и $q_i$, \mbox{$1 \leq i \leq n
$}, суть i-ые левые и i-ые правые дополнительные
члены данного линейного разложения.\\
1. Пусть также для некоторого номера $i$, \mbox{$1
\leq i \leq n$}, и слов $x', x'', u, s, r, h, y$ выполнены условия: \\
$x_i = x'x'' \; \& \; x'' \neq \emptyset$, \\
$xr = x_{1}x_{2} \ldots x_{i-1}x'us =ys\in SF(A, \pi)$, \\
$u = x''h \in D_{\tau}$. \\
Тогда существует линейное разложение $y_1,y_2,\ldots , y_m$ слова
$y$, для которого $(\;y_m = u\;) \; \& \; (\;m\leq i+1\;) \; \& \;
(\;u\sim p_ix_iq_i \Rightarrow m \leq i\;)$.
\\2. Пусть, иначе, для некоторого номера $i$, \mbox{$1
\leq i \leq n$}, и слов $x', x'', u, s, r, h, y$ выполнены условия: \\
$x_i = x'x'' \; \& \; x' \neq \emptyset$, \\
$rx = sux''x_{i+1} \ldots x_{n-1}x_{n} = sy\in SF(A, \pi)$, \\
$u = hx' \in D_{\tau}$. \\
Тогда существует линейное разложение $y_1,y_2,\ldots , y_m$ слова
$y$, для которого $(\;y_1 = u\;) \; \& \; (\;m\leq n-i+2\;) \; \& \;
(\;u\sim p_ix_iq_i \Rightarrow m \leq n-i+1\;)$.
\end{lm}
\begin{proof}[Доказательство]
Пусть выполнены условия пункта 1. Если \mbox{$u\sim p_ix_iq_i $}, то
по утверждению $2.1$ для некоторого слова $f$ выполнено
$$ys=x_{1}x_{2}\ldots x_{i-1}x'us \stackrel{\pi}{=} fp_{i}x'us
\stackrel{\pi}{=} f\underbrace{p_{i}x'p_{i}x'}x''q_{i}s,
$$
и, следовательно, слово $p_ix'$ пусто. Из пустоты слова $p_i$ по
утверждению $2.2$ имеем, что последовательность
\mbox{$x_1,x_2,\ldots,x_{i-1},x_iq_i$} есть линейное разложение
слова \mbox{$x_1x_2\ldots x_{i-1}x_iq_i$}. Далее, из предположения
\mbox{$u\sim p_ix_iq_i =x_iq_i$} по утверждению $2.3$ имеем, что
последовательность \mbox{$x_1,x_2,\ldots,x_{i-1},u$} есть линейное
разложение слова \mbox{$y=x_1x_2\ldots x_{i-1}u$}. Это доказывает
пункт $1$ в рассматриваемом случае.
\par Пусть теперь \mbox{$\neg (u\sim p_ix_iq_i) $}. Если
слово $q_i$ пусто, то по утверждению $2.2$ последовательность
\mbox{$x_1,x_2,\ldots,x_{i}$} есть линейное разложение слова
\mbox{$x_1x_2\ldots x_{i}$}. Далее, поскольку, очевидно, слово
$x''h$ принадлежит множеству $Lin$ и имеет линейное разложение длины
$1$, последовательность \mbox{$x_1,x_2,\ldots, x_{i-1},x',u$} есть
линейное разложение слова \mbox{$y=x_1x_2\ldots x_{i-1}x'u$} по
утверждению $2.4$, из чего вновь следует пункт $1$ доказываемого
утверждения.
\par Если
же \mbox{$q_i\neq \emptyset$}, то по определению $1$ выполнено
\mbox{$i<n \;\&\; p_{i+1}= \emptyset$}. Покажем, что это приводит к
противоречию. Действительно, утверждению $3$ в этом случае
\mbox{$u\sim p_{i+1}x_{i+1}q_{i+1}=x_{i+1}q_{i+1}$}, откуда в силу
утверждения $2.1$ для некоторого слова $g$ справедлива цепочка
отношений:
$$ ys \stackrel{\pi}{=} x_{1}x_{2}\ldots
x_{i}x_{i+1}q_{i+1}gr \stackrel{\pi}{=} x_{1}x_{2}\ldots x_{i}ugr
\stackrel{\pi}{=} x_{1} x_{2}\ldots x'\underbrace{x''x''}hgr \notin
SF(A, \pi),
$$
\par
Второй пункт утверждения, симметричный первому, доказывается
аналогично.
\end{proof}
\begin{lm} Пусть для некоторых слов $x,y,z,u$ выполнены условия \mbox{$x = yuz \in SF(A, \pi)\cap Lin(n)$} и \mbox{$u
\sim v$}. Тогда слово $yvz$ принадлежит множеству \mbox{$Lin(n)$}.
\end{lm}
\begin{proof}[Доказательство]
Рассмотрим линейное разложение $x_{1}, x_{2}, \ldots , x_{n}$ слова
$x$. Тогда существуют такие номера $i, j$, \mbox{$1 \leq i \leq j
\leq n$}, и слова $x'_i, x''_i, x'_j, x''_j$, что
$$
x_i = x'_ix''_i \; \& \; x''_i \neq \emptyset \; \& \; y = x_{1}
 \ldots x_{i-1}x'_i \; \& $$
$$ \; x_j = x'_jx''_j \; \& \; x'_j \neq
\emptyset \; \& \; z = x''_{j} x_{j+1} \ldots x_{n}.
$$
Если \mbox{$i=j$}, то утверждение немедленно следует из утверждения
$2.3$. В противном случае из утверждения $4$ следует существование
линейного разложения \mbox{$y_1,y_2,\ldots,y_{m_1-1},y_{m_1} $}
слова $yu$, для которого \mbox{$y_{m_1}=u $} и \mbox{$m_1 \leq
i+1$}, а также линейного разложения
\mbox{$z_1,z_2,z_3,\ldots,z_{m_2}$} слова $uz$, для которого
\mbox{$z_1=u$} и \mbox{$m_2 \leq n-j+2$}. Тогда по утверждению $2.3$
последовательность \mbox{$y_1,y_2,\ldots,y_{m_1-1},v$} есть линейное
разложение слова \mbox{$y_1y_2\ldots y_{m_1-1}v$}, а
последовательность \mbox{$v,z_2,z_3\ldots,z_{m_2}$} есть линейное
разложение слова \mbox{$vz_2z_3\ldots z_{m_2}$}. Тогда из
утверждения $2.4$ следует, что слово \mbox{$yvz=y_1y_2\ldots
y_{m_1-1}vz_2z_3\ldots z_{m_2}$} принадлежит множеству \mbox{$
Lin(m_1+m_2-1)$}. Очевидно, при \mbox{$j\geq i+2$} это влечет
включение \mbox{$yvz\in Lin(n)$}. Если же \mbox{$j=i+1$}, то по
утверждению $3$ имеет место \mbox{$u \sim p_ix_iq_i \;\vee\; u \sim
p_jx_jq_j$}, что по утверждению $4$ влечет выполненность одного из
неравенств \mbox{$m_1\leq i$} и \mbox{$m_2\leq n-i$}. Это вновь
приводит к справедливости доказываемого утверждения.
\end{proof}
\begin{sss}
Класс эквивалентности $[w]_{\pi}$ каждого линейно разложимого
порядка $n$ и бесквадратного относительно системы соотношений $\pi$
слова $w$ конечен и состоит только из линейно разложимых слов
порядка $n$.
\end{sss}
\begin{lm} Пусть для некоторых слов $x,y,e$ выполнены условия \mbox{$xe \in Lin$}, \mbox{$ey \in Lin$} и
\mbox{$xey \in SF(A, \pi)$}. Тогда слово \mbox{$xey$} принадлежит
множеству \mbox{$Lin$}.
\end{lm}
\begin{proof}[Доказательство] Если слово $e$ пусто, утверждение есть
частный случай утверждения $2.4$. Если \mbox{$e\neq\emptyset$},
рассмотрим линейное разложение \mbox{$y_{1}, y_{2}, \ldots , y_{m}$}
слова $ey$. Тогда для некоторого номера $i$, \mbox{$1\leq i\leq m$},
и слов \mbox{$y', y'', e'$} выполнено:
$$ e = y_1y_2\ldots y_{i-1}y' \; \& \; y_i = y'y'' \; \& \; y = y''y_{i+1}y_{i+2}\ldots
y_{m} \; \& \; y'\neq \emptyset.$$ По утверждению $2.1$ существуют
такие слова $f$ и $g$ что $$y_1y_2\ldots y_{i-1}\stackrel{\pi}{=}
fp_i\,\,\text{и}\,\, y_{i+1}y_{i+2}\ldots y_{m}\stackrel{\pi}{=}
q_ig,$$ где $p_i$ и $q_i$ \--- соответственно $i$-ый левый и $i$-ый
правый дополнительные члены указанного линейного разложения.
Заметим, что из этого следует, что слова \mbox{$p_iy_iy_{i+1} \ldots
y_m $} и \mbox{$xfp_iy_iq_ig$} (и все их подслова) бесквадратны
относительно системы соотношений \mbox{$\pi$}.
\par По следствию из утверждения $5$ слово
\mbox{$xfp_iy'\stackrel{\pi}{=}xe$} принадлежит множеству $Lin$.
Покажем, что
$$
xfp_iy_{i}q_{i}\in Lin,
$$
причем некоторое линейное разложение слова \mbox{$xfp_iy_{i}q_{i}$}
имеет последний член \mbox{$p_iy_{i}q_{i}$}.
\par Действительно, выше замечено включение
\mbox{$xfp_iy_{i}q_{i}\in SF(A,\pi)$}. Рассмотрим линейное
разложение \mbox{$x_{1}, x_{2}, \ldots , x_{n}$} слова
\mbox{$xfp_iy'$}. Поскольку слово $p_iy'$ непусто, существует такой
номер $j$, \mbox{$1\leq j\leq n$}, что
$$
x_j=x'x''\,\&\, x''\neq\emptyset\,\&\,xfp_iy'y''q_{i}=x_1x_2\ldots
x_{j-1}x'p_ix_{i}q_{i},
$$
и данный факт немедленно следует из утверждения $4$, так как
\mbox{$p_ix_{i}q_{i}\in D_{\tau}$}.
\par Рассмотрим теперь слово \mbox{$p_iy_iy_{i+1} \ldots y_m $}. Оно принадлежит
множеству $Lin$ утверждению $2.2$ и, как замечено выше, бесквадратно
относительно системы соотношений $\pi$. Поэтому по следствию из
утверждения $5$ множество $Lin$ содержит слово
\mbox{$p_iy_iq_ig\stackrel{\pi}{=}p_iy_iy_{i+1} \ldots y_m$}, причем
слово $p_iy_iq_i$ есть первый член некоторого линейного разложения
слова $p_iy_iq_ig$ (последнее следует, например, из утверждения
$4$).
\par
Тогда слово $xfp_iy_iq_ig$ принадлежит множеству $Lin$ по
утверждению $2.4$, а равное ему относительно системы соотношений
$\pi$ слово $xey$ \--- по следствию из утверждения $5$.
\end{proof}
\par Дальнейшие рассуждения удобно
проводить, используя понятие \emph{вхождение} (см.
~\cite{adan}).\par Пусть дан алфавит $A$. Вхождением слова $e$ в
алфавите $A$ в слово \mbox{$x = peq$} в том же алфавите называется
слово \mbox{$p\ast e \ast q$} в алфавите \mbox{$A \cup \{\, \ast \,
\}$}, где символ \mbox{$\ast$} не принадлежит алфавиту $A$. Слово
$e$ называется \emph{основой} вхождения \mbox{$p\ast e \ast q$}.
Пусть для некоторых слов $q$, $p$, $e$, $q$, $r$, $d$, $s$ в
алфавите $A$ имеют место равенства \mbox{$x = peq = rds$}. Будем
говорить, что вхождение \mbox{$p \ast e \ast q$} \emph{содержится}
во вхождении \mbox{$r \ast d \ast s$}, если \mbox{$|r| \leq |p|$} и
\mbox{$|s| \leq |q|$}. Вхождения \mbox{$p \ast e \ast q$} и \mbox{$r
\ast d \ast s$} слов $e$ и $d$ в слово \mbox{$x = peq = rds$}
\emph{пересекаются}, если найдется некоторое вхождение \mbox{$v \ast
f \ast w$} непустого слова $f$ в то же слово \mbox{$x = vfs$},
которое содержится одновременно во вхождениях \mbox{$p \ast e \ast
q$} и \mbox{$r \ast d \ast s$}. Максимальное по длине основы из
таких вхождений называется \emph{пересечением} вхождений \mbox{$p
\ast e \ast q$} и \mbox{$r \ast d \ast s$}. \emph{Объединением}
вхождений называется минимальное по длине основы вхождение, в
котором эти вхождения содержатся.
\begin{df} Вхождение $\varphi$ слова \mbox{$e \in
Lin$} в слово \mbox{$w \in A^\ast$} будем называть
\emph{максимальным}, если оно не содержится ни в каком отличном от
себя вхождении с линейно разложимой основой. Для каждого
\mbox{$n<\omega$} множество всех максимальных вхождений $\varphi$
различных слов \mbox{$e\in Lin(n)$} в слово $w$ будем обозначать
символом \mbox{$MaxLin( n, w )$}.
\end{df}
\par Из утверждения $6$ следует следующее
\begin{zam}
Если некоторое максимальное вхождение $\varphi$ слова \mbox{$e \in
Lin$} в слово \mbox{$w \in SF(A, \pi)$} пересекается с некоторым
вхождением слова $\psi$ слова \mbox{$e' \in Lin$} в слово $w$, то
вхождение $\psi$ содержится во вхождении $\varphi$. Таким образом,
максимальные вхождения линейно разложимых слов в бесквадратное
относительно системы соотношений $\pi$ слово не пересекаются.
\end{zam}
Из последнего факта следует, что каждое слово \mbox{$w \in SF(A,
\pi)$} может быть представлено в виде:
$$ w = r_1x_1r_2x_2 \ldots r_nx_nr_{n+1} \; (\, 0 \leq n < \omega \,),$$ где для всех $i$, \mbox{$1 \leq i \leq
n $}, слова $r_i$ не содержат определяющих подслов, а каждое из
вхождений \mbox{$r_1x_1r_2x_2 \ldots r_i\ast x_i\ast r_{i+1}\ldots
r_nx_nr_{n+1}$} принадлежит множеству \mbox{$MaxLin(n_i, w)$} для
некоторых натуральных чисел $n_i$, \mbox{$0 < n_i < \omega $}.
Считая натуральные числа $n$ и $n_i$, \mbox{$1\leq i\leq n$},
фиксированными, определим множество $T_w \rightleftharpoons
\\ \{ \,r_1y_1r_2y_2 \ldots r_ny_nr_{n+1} \mid y_i \in Lin(n_i) \;
\& \; r_1y_1r_2y_2 \ldots r_i\ast y_i\ast r_{i+1}\ldots
r_ny_nr_{n+1} \in MaxLin(\, n_i,
r_1y_1r_2y_2 \ldots r_ny_nr_{n+1} \,), \; 1 \leq i \leq n\, \}.$\\
По утверждению $2.5$ это множество конечно. Таким образом, для
доказательства теоремы достаточно доказать следующее утверждение:
$$ \forall w \in A^\ast \; w \in SF(A,\pi) \Rightarrow [w]_\pi \subseteq
T_w. \eqno (\ast \ast)$$ Пусть слово \mbox{$rus = r_1y_1r_2y_2
\ldots r_ny_nr_{n+1}$} принадлежит множеству \mbox{$T_{w}\cap SF(A,
\pi)$}, и для каждого номера $i$, \mbox{$1 \leq i \leq n$}, слово
$r_i$ не содержат определяющих слов, а вхождение \mbox{$r_1x_1r_2x_2
\ldots r_i\ast x_i\ast r_{i+1}\ldots r_nx_nr_{n+1}$} принадлежит
множеству \mbox{$MaxLin(n_i, w)$}. Пусть выполнено
\mbox{$v\stackrel{\pi}{\leftrightarrow} u$}. Докажем, что слово
\mbox{$rvs$} принадлежит множеству \mbox{$T_w$}.
\par Согласно замечанию, вхождение \mbox{$r \ast u \ast s$}
содержится в одном из вхождений \mbox{$r_1y_1r_2y_2 \ldots r_i \ast
y_i \ast r_{i+1} \ldots r_ny_nr_{n+1}$}, т.е. существуют такие слова
$r', s'$, что \mbox{$r = r_1y_1r_2y_2 \ldots r_ir'$} и \mbox{$y_i =
r'us'$}. Тогда по утверждению $5$ имеет место включение \mbox{$r'vs'
\in Lin(n_i)$}. Допустим, что вхождение \mbox{$r_1y_1r_2y_2 \ldots
r_i \ast r'vs' \ast r_{i+1}\ldots r_ny_nr_{n+1}$} не максимально.
Это означает, что для некоторых слов $a, b, c, d$ выполнены
равенства \mbox{$r_1y_1r_2y_2 \ldots r_i = ab$} и
\mbox{$r_{i+1}y_{i+1} \ldots r_ny_n r_{n+1} = cd$}, причем слово
$br'vs'c$ принадлежит множеству $Lin$ и \mbox{$bc \neq \emptyset$}.
Тогда \mbox{$br'us'c \in Lin$} по утверждению $5$, и, следовательно,
\mbox{$r_1y_1r_2y_2 \ldots r_i \ast r'us' \ast r_{i+1}\ldots
r_ny_nr_{n+1} \notin MaxLin(\, n_i,\, rus \,)$}; противоречие.
\par Таким образом,
слово \mbox{$rvs$} принадлежит множеству \mbox{$T_w$}, что
доказывает утверждение $(\ast \ast)$.
\end{proof}

\begin{sss}Проблема распознавания бесквадратности данного слова
относительно системы из двух определяющих соотношений
алглоритмически разрешима.
\end{sss}
\begin{zakzam}
Минимальное количество соотношений, достаточное для построения
копредставления с неразрешимой проблемой распознавания
бесквадратности данного слова, неизвестно. Пример A. Карпи и A. де
Лука содержит более двух тысяч соотношений.
\end{zakzam}

\begin{blagodar}
Автор выражает глубокую благодарность академику РАН профессору Адяну
Сергею Ивановичу за постановку задачи и рекомендации по выбору
основных определений.
\end{blagodar}

\end{document}